\documentclass{amsart}

\usepackage[latin1]{inputenc}
\usepackage{amsfonts}
\usepackage{amsmath}
\usepackage{amsthm}
\usepackage{amssymb}
\usepackage{latexsym}
\usepackage{enumerate}

\newtheorem{theorem}{Theorem}[section]
\newtheorem{lemma}[theorem]{Lemma}
\newtheorem{proposition}[theorem]{Proposition}

\newtheorem{question}[theorem]{Question}

\newtheorem{definition}[theorem]{Definition}

\numberwithin{equation}{section}

\begin{document}

\newcommand{\cc}{\mathfrak{c}}
\newcommand{\bb}{\mathfrak{b}}
\newcommand{\N}{\mathbb{N}}
\newcommand{\Q}{\mathbb{Q}}
\newcommand{\R}{\mathbb{R}}
\newcommand{\K}{\mathbb{K}}
\newcommand{\E}{\mathbb{E}}
\newcommand{\UE}{\mathbb{U}\mathbb{E}}
\newcommand{\RN}{\mathbb{RN}}
\newcommand{\QRN}{\mathbb{QRN}}
\newcommand{\CM}{\mathbb{CM}}

\newcommand{\C}{\mathbb{C}}
\newcommand{\PP}{\mathbb{P}}
\newcommand{\forces}{\Vdash}
\newcommand{\dom}{\text{dom}}
\newcommand{\osc}{\text{osc}}

\title{Universal objects and associations between classes of Banach spaces
and  classes of compact  spaces}

\author{Piotr Koszmider}
\thanks{The  author was partially supported by the National Science Center research grant 2011/01/B/ST1/00657. } 

\address{Institute of Mathematics, Polish Academy of Sciences,
ul. \'Sniadeckich 8,  00-956 Warszawa, Poland}
\email{\texttt{P.Koszmider@impan.pl}}

%\author{}

%\address{d}

%\email{}

\date{}

\begin{abstract} 
In the context of classical associations between classes of  Banach spaces and 
classes of compact Hausdorff spaces we survey known results
and open questions concerning the existence and nonexistence of universal Banach spaces
and of universal compact spaces  in various classes. This gives
quite a  complex network of
interrelations  which quite often depend on additional set-theoretic assumptions.
\end{abstract}
\maketitle

\markright{Universal objects and associations}

\section{Introduction}

In this survey note, we would like to look at universal objects in several classes of compact spaces and 
several classes of Banach spaces. 
All Banach spaces are over the reals and are assumed to be infinite dimensional. All
compact spaces are assumed to be Hausdorff and infinite.

In section 2. we use the Stone duality between totally disconnected compact Hausdorff spaces with continuous maps and Boolean algebras
with homomorphisms as a motivation for seeing the  interactions between compact
Hausdorff spaces with continuous maps and Banach spaces with linear bounded operators. We interpret
what is known in functional analysis as permanence properties in the light of the Stone duality.
We define several notions of associated classes of Banach spaces and compact spaces (associated,
$K$-associated, $B$-associated and strongly associated).
We show elementary facts about the relationships among theses notions and note natural ways of defining 
associated classes.  

In section 3. we list the classes of compact spaces we will be concerned with: uniform Eberlein, Eberlein, Corson, Corson with property M,
Radon-Nikod\'ym, Quasi Radon-Nikod\'ym; and classes of Banach spaces: Hilbert generated, weakly compactly generated, weakly Lindel\"of,
weakly Lindel\"of detremined, Asplund generated, subspaces of Asplund generated.
We distinguish subclasses of given density character. We also state known relationships  among these classes.

In section 4. we describe various types of universality (surjective, injective, isomorphic, isometric, weak) and prove
elementary results on how the existence of universal objects for some class gives the existence of the universal objects
in an associated class depending on the type of the association and type of universality. 
We note that unless we are in the case
of weakly universal Banach spaces, the existence of various types of universal compact spaces imply the
existence of universal Banach spaces but apparently not vice versa.  In particular all 
present in the literature proofs of results on the nonexistence of universal Banach spaces are considerably harder
than the results on the nonexistence of universal compact spaces.

Finally in section 5. we list in a systematic way known results and open problems on the existence and nonexistence of
universal compact spaces and Banach spaces in the classes introduced in section 3. We stress general patterns along
the various notions of association and duality from section 2., in the context of various notions of universality from section 4.

In section 6 we mention known results concerning the universality of $l_\infty/c_0$.  Throughout the paper 
we  mention many questions which seem to be open.
For undefined notions we refer the reader to the standard textbooks like \cite{fabianetal, koppelberg, negrepontis, engelking}. We are grateful to Antonio Avil\'es for conversations which allowed the author to improve this
paper.

\section{Associations between compact Hausdorff spaces and Banach spaces}

As a motivation for the associations between compact spaces and Banach spaces
we would like to propose the Stone duality between compact totally disconnected spaces and
Boolean algebras. This duality consists of two standard constructions of  contravariant functors.

Recall that to each Boolean algebra $\mathcal A$ we can associate
a totally disconnected compact space 
$$K_{\mathcal A}=\{h: h\  \hbox{\rm is a homomorphism from}\  \mathcal A \ \hbox{\rm into}\  \{0,1\}\}$$
with the smallest topology making all evaluations of elements of $K_{\mathcal A}$
at  a fixed element of $\mathcal A$ continuous. This topology is called the Stone topology (see e.g. \cite{koppelberg}).
Then, to every homomorphism of Boolean algebras $g:\mathcal B\rightarrow \mathcal A$ we
associate a continuous map
$\phi_g: K_{\mathcal A}\rightarrow K_{\mathcal B}$ given for every $h\in K_{\mathcal A}$ by
$$\phi_g(h)=h\circ g.$$
This way, by a standard argument, $g$ is onto if and only if $\phi_g$ is an injection and vice versa.
Also to each compact totally disconnected space $K$ we can associate a Boolean algebra
$$\mathcal A_K=\{\phi: \phi \ \hbox{\rm is a continuous function from}\ K\  \hbox{\rm into}\ \{0,1\}\}$$
with the operations of taking the minimum, taking the maximum and the difference with the constant function equal to one.  Given a continuous function $\psi: L\rightarrow K$, where $K$ and $L$ are
compact totally disconnected, we can define $h_\psi:\mathcal A_K\rightarrow \mathcal A_L$ by
$$h_\psi(\phi)=\phi\circ \psi.$$
We have similar relations among surjections and injections.
The Stone duality gives us that $K_{\mathcal A_L}$ is homeomorphic to
$L$ for every totally disconnected compact $K$ and $\mathcal A_{K_{\mathcal B}}$ is isomorphic to
$\mathcal B$ for every Boolean algebra establishing a complete correspondence between the two categories which is
elegant and very fruitful for both the topology and Boolean algebras.  
Taking continuous images of totally disconnected compact spaces corresponds to taking Boolean subalgebras
and taking closed subspaces corresponds to taking quotients of Boolean algebras.

Now let us see what happens if in an analogous way, instead of a Boolean algebra, we associate 
with a given compact Hausdorff space $K$ 
a Banach space.
We consider
$$C(K)=\{f: f\  \hbox{\rm is a continuous function from}\ K\   \hbox{\rm into}\ \R\}$$
with the pointwise linear operations. 
Given a continuous function $\psi: L\rightarrow K$ where $K$ and $L$ are
compact  we can define
$$T_\psi(f)=f\circ \psi .$$
We have similar relations among surjections and injections as in the case of Boolean algebras.
Now, to each Banach space $X$ we associate the dual ball with the weak$^*$ topology, that is
$$B_{X^*}=\{x^*: x^*\  \hbox{\rm is linear and bounded map from}\  X \ \hbox{\rm into}\  \R\ \hbox{\rm 
with}\  ||x^*||\leq 1\}$$
with the smallest topology  making all evaluations of elements of $B_{X^*}$
at a fixed  element of $X$ continuous. This topology is called the weak$^*$ topology (see e.g. \cite{fabianetal} for details).
Here the analogy to the Stone duality ends, because given a linear bounded operator
$T: Y\rightarrow X$, if we
define
$T^*: X^*\rightarrow Y^*$ by
$$T^*(x^*)=x^*\circ T,$$
for every $x^*\in X^*$  there is no reason why $T^*$ restricted to $B_{X^*}$ has to end up in $B_{Y^*}$ unless
$T$ is an isometry. Now we have two options, either consider only isometries, but this is not what is done in
Banach space theory, or consider the duality between $X^*$ and $Y^*$ but then we loose compactness and all of its privileges.

The third way\footnote{For possible functors
 between various categories of compact spaces and various categories of Banach spaces see \cite{semadenifunktor}.} which is behind the theory discussed in this paper is to note that
$T^*$ is continuous with respect to the weak$^*$ topologies, so the image of $B_{X^*}$ must be
compact, and so bounded in $Y^*$,
hence there is $n\in \N$ such that $T^*[B_{X^*}]\subseteq n B_{Y^*}$ and of course $n B_{Y^*}$ 
is a homeomorphic copy of $n B_{Y^*}$.  Then some analogies with the 
behaviour of a contravariant functor can be saved:

\begin{proposition}\label{dualities}
Suppose $K,L$ are compact spaces, $\psi$ is a continuous map, $X, Y$ are Banach spaces
and $T$ is a linear operator.
\begin{enumerate}
\item If $K$ is a continuous image of $L$, then $C(K)$ is isometric to a subspace of $C(L)$
\item If $L$ is a subspace of $K$, then $C(L)$ is isometric to a quotient of $C(K)$
\item If $X$ is a subspace of $Y$, then $B_{X^*}$ is a continuous image of $B_{Y^*}$
\item If $X$ is isomorphic to a quotient of $Y$, then $B_{X^*}$ is homeomorphic to a subspace of $B_{Y^*}$
\end{enumerate}
\end{proposition}

It should be nevertheless noted, that even isomorphic Banach spaces may have non-homeomorphic dual balls
(with the weak$^*$ topologies), this can even happen in the very canonical case of a nonseparable Hilbert space
(see \cite{antoniohilbert}, here the weak and the weak$^*$ topologies agree). 
On the other hand
we have the following analogy of the Stone duality:

\begin{proposition}\label{semistone} Let $K$ be a compact space and $X$ be a Banach space.
\begin{enumerate}
\item  $K$ is homeomorphic to a subspace of $B_{C(K)^*}$
\item  $X$ is isometric to a subspace of $C(B_{X^*})$
\end{enumerate}
\end{proposition}

\begin{proof}

If $x$ is a point of a compact space $K$ then the
 pointwise measure functionals $\delta_x$ are inside the dual ball of the Banach space $C(K)$
and  the map that sends 
$x\in K$ to $\delta_x$ is a homeomorphism.
For any Banach space
 the map which sends $x$ to $x^{**}$ in the
bidual is an isometry whose image is included in $C(B_{X^*})$.
\end{proof}

However, unlike in the case of the Stone duality, $B_{C(K)^*}$ is not homeomorphic to $K$ nor
$C(B_{X^*})$ is not isomorphic to $X$ in most cases.  

Note that in \ref{dualities} isomorphic versions of (3) is not true: if $T:X\rightarrow Y$
is an isomorphism onto its range, then $B_{X^*}$ may not be a continuous image of $B_{Y^*}$.
This follows from results of A. Avil\'es (Theorems 2 and 4 or \cite{antoniohilbert}) even in the case
of both of the spaces isomorphic to a Hilbert space. However we have the following:

\begin{proposition} Suppose that $X$ and $Y$ are Banach spaces and  $T:X\rightarrow Y$
is an isomorphism onto its range. Then the family of  compact subsets of $X^*$ with the weak$^*$ topology
is included in the family of continuous images of  compact subsets of $Y^*$ with the weak$^*$ topology.
\end{proposition}
\begin{proof} Let $T^{-1}:T[X]\rightarrow X$. $T^*: Y^*\rightarrow X^*$ is weak$^*$ continuous.
Note that 
$$  {1\over{||T^{-1}||}}B_{X^*} \subseteq T^*[B_{Y^*}]\subseteq ||T||B_{X^*}.$$
The right hand side inclusion is clear from the fact that $||T||=||T^*||$. For the 
left hand side inclusion note that if $\phi\in {1\over{||T^{-1}||}}B_{X^*}$, then
$(T^{-1})^*(\phi)\in B_{T[X]^*}$ is of at most norm one, and so it has
an extension $\psi\in B_{Y^*}$. Now $T^*(\psi)=\psi\circ T=\phi\circ T^{-1}\circ T=\phi$.

If $K$ is compact in $X^*$, it is bounded and so $K\subseteq nB_{X^*}$ for some $n\in\N$.
It follows that $K$ is a continuous image under $T^*$ of $(T^*)^{-1}[K]\subseteq {n{||T^{-1}||}}
B_{Y^*}$.
\end{proof}

So, it may be more natural to consider classes of compact spaces associated with a  Banach space
(all compact weak$^*$ subsets of its dual) than just the dual ball. This leads, however, to classes of Banach spaces if we aim at some duality. In the  Banach space theory practice one considers quite general classes of Banach
spaces and compact spaces
which are closed by taking these dual objects. In most cases the classes of compact spaces are closed under taking subspaces and so the family of all weakly$^*$ compact subsets of  a given dual $X^*$ is included in such a class
if and only if the dual ball $nB_{X^*}$ belongs to the class.
Here by classes of compact spaces or Banach spaces we mean
classes which are homeomorphism or isomorphism invariant respectively:

\begin{definition}\label{associated} Let $\mathcal K$ be a class of compact spaces 
and $\mathcal B$ be a class of Banach spaces. We say that $\mathcal K$ and $\mathcal B$
are associated if and only the following two implications hold:
\begin{itemize}
\item If $K\in \mathcal K$, then $C(K)\in \mathcal B$.
\item If $X\in \mathcal B$, then $B_{X^*}\in \mathcal K$.
\end{itemize}
If instead of the first implication we have equivalence, we say that the classes are
$K$-associated. If instead of the second implication we have equivalence, we say that the classes are
$B$-associated. If instead of  both of the
implications above we have equivalences, we say that
the classes are strongly associated.
\end{definition}

\begin{lemma}\label{KBassociated}
Suppose $\mathcal K$ is a class of compact spaces and $\mathcal B$ is a class of
Banach spaces such that $\mathcal K$ and $\mathcal B$ are associated.
\begin{enumerate}
\item If $\mathcal K$ is closed under taking subspaces, then $\mathcal K$ and $\mathcal B$ are $K$-associated.
\item If $\mathcal B$ is closed under taking subspaces, then $\mathcal K$ and $\mathcal B$ are $B$-associated.
\end{enumerate}
\end{lemma}
\begin{proof}
If $C(K)\in \mathcal B$, by the second implication of \ref{associated}, $B_{{C(K)}^*}\in \mathcal K$.
By \ref{semistone} $K$ is a subspace of $B_{C(K)^*}$, so, by the hypothesis on $\mathcal K$, we get
$K\in \mathcal K$. The other part is analogous.
\end{proof}

\begin{lemma}\label{heredityassoc}
Suppose $\mathcal K$ is a class of compact spaces closed under taking
continuous images and $\mathcal B$ is a class of
Banach spaces which is not closed under taking subspaces and suppose  that $\mathcal K$ and $\mathcal B$ are associated.
Then $\mathcal K$ and $\mathcal B$ are not $B$-associated.
\end{lemma}
\begin{proof}
Let $X\not\in \mathcal B$ be a subspace of $Y\in\mathcal B$. Then $B_{X^*}$ is a continuous 
image of $B_{Y^*}\in \mathcal K$ by \ref{dualities}, hence  $B_{X^*}\in \mathcal K$ which proves
that the second implication of \ref{associated} cannot be reversed.
\end{proof}

We note here that some natural classes of compact spaces are not associated to
classes of Banach space. For example
$B_{X^*}$ is never totally disconnected, so neither totally disconnected spaces nor scattered spaces
are associated with a class of Banach spaces. 
Also many natural classes of Banach spaces are not associated with classes of compact spaces, for example reflexive spaces are not
associated with a class of Banach spaces because no infinite dimensional $C(K)$ space is reflexive.
As the above classes play important roles, the issues of universality are well investigated for them, 
e.g. see \cite{szlenk, wojtaszczyk, asplunduniversal, bellwitek}, but we will not be concerned with them in this note.
In general we have the following:

\begin{proposition}\label{associatedcrit} A class of compact spaces $\mathcal K$
closed by taking subspaces is associated with
a class of Banach spaces if and only if $B_{C(K)^*}\in \mathcal K$ whenever
$K\in \mathcal K$. A class of Banach spaces  $\mathcal B$ closed under taking isometric
copies is associated with
a class of compact spaces if and only if $C(B_{X^*})\in \mathcal B$ whenever
$X\in \mathcal B$.
\end{proposition}
\begin{proof} If the above conditions hold define 
$$\mathcal B=\{ X: B_{X^*}\in \mathcal K\}, \ \ \ \mathcal K=\{ K: C(K)\in \mathcal B\},$$
respectively. It remains to prove that $\mathcal B$ is isomorphism invariant and
$\mathcal K$ is homeomorphism invariant. In the first case we use the fact that
if we have two isomorphic Banach spaces, then the dual ball of one is a subspace of a homeomorphic
copy of the dual ball of the other. In the second case we use the fact that 
if two compact spaces are homeomorphic, then the spaces of their continuous functions are isometric.
\end{proof}

There are also elementary ways of obtaining strongly associated classes:

\begin{proposition}\label{makingstrongly} Suppose that $\mathcal B$ and $\mathcal K$ are associated classes
of Banach spaces and compact spaces respectively. Suppose $\mathcal K$ is closed under taking subspaces of its elements. Let $\mathcal B'$ be the class of all
subspaces of elements of $\mathcal B$ and $\mathcal K'$ be the class of continuous images of
elements of $\mathcal K$.
 Then $\mathcal B'$ and $\mathcal K'$ are strongly associated.
\end{proposition}
\begin{proof} As $\mathcal B'$ and $\mathcal K'$ are closed under taking subspaces, by
\ref{KBassociated} it is enough to
prove that the classes are associated.
Suppose $Y\in  \mathcal B'$, i.e, $X\subseteq Y$ with $Y\in \mathcal B$, 
then $B_{Y^*}$ is a continuous image od $B_{X^*}\in \mathcal K$ by \ref{dualities}, 
and hence $B_{Y^*}\in \mathcal K'$. If $L\in \mathcal K'$, then $L$ is a continuous image
of some $K\in \mathcal K$ and so $C(L)\subseteq C(K)\in \mathcal B$ by \ref{dualities}.
\end{proof}

It should be noted that given a class of compact spaces or a class of Banach spaces
an associated class is not uniquely determined. For example WCG Banach spaces and
subspaces of WCG Banach spaces are associated with the class of Eberlein compacta
(see the next section). The intersection of two associated classes with a given class is again
an associated class, so we can talk about minimal associated classes. However
these are pretty trivial and uninteresting, e.g., given associated classes $\mathcal K$
and $\mathcal B$ 
consider $\mathcal B'$ to be equal to the class of all Banach spaces  in $\mathcal B$
isomorphic to  a space of the form $C(K)$ for $K\in \mathcal K$. 
$\mathcal B'$ is the minimal class associated with $\mathcal K$. As our paradigmatic link
between classes of Banach spaces and compact spaces will be obtaining a universal 
Banach space for the associated class from a universal compact space in a given class (or in general
getting an information about Banach spaces from compact spaces)
we are rather interested in a maximal class of Banach spaces associated with minimal
classes of compact spaces. 

\section{Classes of compact and Banach spaces}

 \begin{definition}\label{classesK} Let $\kappa$ be a cardinal. We will use the following notation:
\begin{itemize}
\item $\UE_\kappa$ - uniform Eberlein compact spaces (i.e., compact subspaces of Hilbert spaces
with the weak topology) of weight $\leq\kappa$,
\item $\E_\kappa$ -  Eberlein compact spaces (i.e., compact subspaces of Banach spaces
with the weak topology) of weight $\leq\kappa$,
\item $\C_\kappa$ - Corson compact spaces (i.e., compact subspaces of $\Sigma$-products of
$\R$) of weight $\leq\kappa$,
\item $\CM_\kappa$ - Corson compact spaces with property M (i.e., 
Corson compacta where every Radon measure on $K$ has a separable support) of weight $\leq\kappa$,

\item $\RN_\kappa$ - Radon-Nikod\'ym compact spaces (i.e., fragmented by a lower semi-continuous metric)  of weight $\leq\kappa$,
\item $\QRN_\kappa$ -  Quasi Radon-Nikod\'ym compact spaces (i.e., fragmented by a lower-semi continuous quasi-metric)  of weight $\leq\kappa$,
\item $\K_\kappa$ - all compact spaces of weight $\leq\kappa$.
\end{itemize}
\end{definition}

We have $\UE_\kappa\subseteq \E_\kappa\subseteq  \CM_\kappa\subseteq\C_\kappa\subseteq \K_\kappa$
and  $\E_\kappa\subseteq \RN_\kappa\subseteq \QRN_\kappa\subseteq \K_\kappa$
for each cardinal $\kappa$ and $\E_\kappa=\C_\kappa\cap \RN_\kappa$ by \cite{orihuelaetal}.
Note that metrizability in the class of compact spaces is equivalent to countable weight, so
actually the class $\K_\omega$ is the class of metrizable compact spaces.
More details concerning these classes of spaces can be found in \cite{negrepontis} or in \cite{overclasses}.
In particular, all the classes above are closed under taking subspaces and with the exception of
Radon-Nikod\'ym compacta, are closed under taking continuous images (see \cite{rn}).
The first of these facts is trivial and second relays on the results of Benyamini, Rudin and Wage
from \cite{BRW} and Argyros, Mercourakis, Negrepontis from \cite{argyroscorson}.
By 3.2. (3) of \cite{argyroscorson} we have that $\C_\kappa=\CM_\kappa$ for
for every  cardinal $\kappa$ if Martin's axiom a the negation of the continuum hypothesis CH holds.
If CH holds then  $\C_\kappa\not=\CM_\kappa$ for any uncountable $\kappa$ (3. 12. \cite{argyroscorson})
By a result of A. Avil\'es (\cite{antoniob}) $\QRN_\kappa=\RN_\kappa$ if $\kappa<\bb$,
however $\QRN_{2^\omega}\not=\RN_{2^\omega}$ by the main resut of \cite{rn}.

 \begin{definition}\label{classesB} Let $\kappa$ be a cardinal. We will use the following notation:
\begin{itemize}
\item $\overline{\mathcal{H}}_\kappa$ - 
Hilbert generated Banach spaces  of density character $\leq\kappa$
\item $\mathcal{WCG}_\kappa$ -  Weakly compactly generated Banach spaces of density character $\leq\kappa$
\item $\mathcal{WLD}_\kappa$ - Weakly Lindel\"of determined Banach spaces  of density character $\leq\kappa$
\item $\mathcal{L}_\kappa$ - Banach spaces  of density character $\leq\kappa$ which are Lindel\"of in the weak topology. 
\item $\mathcal{AG}_\kappa$ - Asplund generated Banach spaces of of density character $\leq\kappa$ 
\item $\mathcal{SAG}_\kappa$ - subspaces of Asplund generated Banach spaces of of density character $\leq\kappa$ 
\item $\mathcal B_\kappa$ - all Banach spaces of density character $\leq\kappa$.
\end{itemize}
\end{definition}

Good sources of information about the above classes of Banach spaces are:
\cite{fabianetal} for $\overline{\mathcal{H}}_\kappa$, $\mathcal{L}_\kappa$ and for  $\mathcal{WCG}_\kappa$;
 \cite{kalendasurvey} for $\mathcal{WLD}_\kappa$; \cite{overclasses} for $\mathcal{AG}_\kappa$  and $\mathcal{SAG}_\kappa$.  We have 
$\overline{\mathcal{H}}_\kappa\subseteq \mathcal{WCG}_\kappa\subseteq \mathcal{WLD}_\kappa\subseteq\mathcal B_\kappa$
and $\mathcal{WCG}_\kappa\subseteq\mathcal{AG}_\kappa\subseteq \mathcal{SAG}_\kappa\subseteq\mathcal{B}_\kappa$.
All the classes are closed under taking quotients. An example of Rosenthal from \cite{rosenthal} 
is a space from $\overline{\mathcal{H}}_{2^\omega}$ whose subspace is not in $\mathcal{WCG}_{2^\omega}$.
We have the following

\begin{theorem}\label{associationlist} The pairs of classes 
\begin{enumerate}
\item $\UE_\kappa$ and $\overline{\mathcal{H}}_\kappa$ are $K$-associated and are not strongly associated.
\item $\E_\kappa$ and  $\mathcal{WCG}_\kappa$,  are $K$-associated and are not strongly associated.
\item $\mathcal L_{\omega_1}$ is not associated with $\C_{\omega_1}$

\item $\RN_\kappa$ and $\mathcal{AG}_\kappa$ are $K$-associated, strongly
associated if $\kappa<\bb$ and not
strongly associated if $\kappa\geq \bb$;
\item $\QRN_\kappa$ and $\mathcal{SAG}_\kappa$ are strongly associated
\item   $\K_\kappa$ and $\mathcal B_\kappa$ are strongly associated.

\end{enumerate}
\end{theorem}
\begin{proof}
The equivalence of $K\in \UE_\kappa$ and $C(K)\in \overline{\mathcal{H}}_\kappa$ is the result 
of Benyamini and Starbird from \cite{benyaministar}. The implication from
$X\in \overline{\mathcal{H}}_\kappa$ to $B_{X^*}\in UE_\kappa$ is proved in \cite{hajeketal} as Theorem 6.30. 
Since continuous images of uniformly Eberlein compact space are
uniformly Eberlein (\cite{BRW}), \ref{heredityassoc} implies that  $\UE_\kappa$ and $\overline{\mathcal{H}}_\kappa$ 
are not strongly associated since the example of \cite{rosenthal} shows that
$\overline{\mathcal{H}}_\kappa$ is not closed under taking subspaces.
Facts on  $\E_\kappa$ and  $\mathcal{WCG}_\kappa$ can be found e.g. in \cite{fabianetal},
and are due to Amir and Lindenstrauss (\cite{amirlind}).
The fact that the classes are not strongly associated follows as for the previous class.
This gives (1) and (2).

For (3) consider the example of R. Pol (\cite{pollindelof}) of a scattered compact $K$ such that
$C(K)$ is in $\mathcal L_{\omega_1}\setminus \mathcal{WCG}_{\omega_1}$.
By (2) $K$  cannot be an Eberlein compact, so by a result of Alster (\cite{alster})
$K$ is not Corson compact, and so $B_{C(K)^*}$ is not, however $C(K)\in \mathcal L_{\omega_1}$
as proved in \cite{pollindelof}.

For associations of $\RN_\kappa$ with $\mathcal{AG}_\kappa$ and
$\QRN_\kappa$ with $\mathcal{SAG}_\kappa$ see \cite{overclasses}.
As $\mathcal{SAG}_\kappa$, $\QRN_\kappa$ and $\RN_\kappa$ are closed under subspaces,
\ref{KBassociated} implies that the association in (5) is strong and
in (4) it is a $K$-association. By a result of A. Avil\'es (\cite{antoniob})
 $\QRN_\kappa=\RN_\kappa$ if $\kappa<\bb$.  The strong association in (4)
for Banach spaces of density $<\bb$ is Theorem 5 of \cite{antoniob}.
It is also shown in \cite{antoniob} (bottom of page 79) that there is a
WCG Banach space $Y$ with a subspace $X$ of density character $\bb$ 
which is not Asplund generated.  It follows that $B_{X^*}\in \RN_\bb$
because it is an Eberlein compact as a continuous image of an Eberlein compact $B_{Y^*}$,
but $X$ is not Asplund generated, so $\RN_\bb$ and $\mathcal{AG}_\bb$ are not
$B$-associated.
\end{proof}

The situation for Corson compacta and weakly Lindel\"of determined spaces is more
complicated because fundamental relations depend on extra set-theoretic axioms:
$ $\par
$ $\par
\begin{theorem}\label{associationlist2}
$ $\par
\begin{enumerate}
\item $K\in \CM_\kappa$ if and
only if $C(K)\in \mathcal{WLD}_\kappa$ 
\item $X\in \mathcal{WLD}_\kappa$ if
and only if $B_{X^*}\in \C_\kappa$,
\item Assuming MA$+\neg$CH the classes $\C_\kappa=\CM_\kappa$ and $\mathcal{WLD}_\kappa$ are
strongly associated for any uncountable cardinal $\kappa$.
\item Assuming $CH$ 
\begin{enumerate}
\item $\C_{2^\omega}$ is not associated with any class of Banach spaces
\item $\mathcal{L}_{2^\omega}$ is not associated with any class of compact spaces
\item $\mathcal{WLD}_{2^\omega}$ is not associated with any class of compact spaces
\end{enumerate}
\item $\CM_{\kappa}$ is  associated with a class of Banach spaces for any uncountable
cardinal $\kappa$;
\end{enumerate}
\end{theorem}
\begin{proof}
(1) is 5.2. and (2) is 4.17 of \cite{kalendasurvey}.
(3) follows from the above and the fact that $\C_\kappa=\CM_\kappa$ under
MA$+\neg$CH (3.2. (3) \cite{argyroscorson})

The results (4) under CH are proved based on two  examples, of Argyros, Mercourakis and Negrepontis of  a Corson
compact $K_0\in \C_{2^\omega}$ which does not have property M
(3.12  \cite{argyroscorson}) and its strengthening  (4.4. \cite{plebanekcorson}) due to Kalenda based on
examples of G. Plebanek from \cite{plebanekcorson} of a Banach space $X_0$ whose dual
ball is a Corson compact without property M.

Using the fact that for a Corson compact $K\in \C_\kappa$ having property $M$ is equivalent to
$C(K)\in \mathcal L_{\kappa}$ (3.5. \cite{argyroscorson}), we conclude that
$C(K_0)\not \in \mathcal L_{2^\omega}, \ \ C(B_{X_0}^*)\not \in \mathcal L_{2^\omega}$.
And so by the above equivalences we have 
$$K_0\in \C_{2^\omega}, \ \  B_{C(K_0)^*}\not\in \C_{2^\omega},\ \ 
X_0\in \mathcal{WLD}_{2^\omega}, \ \ C(B_{X_0})\not \in \mathcal L_{2^\omega}.$$
So using the fact that $\mathcal{WLD}_{2^\omega}\subseteq \mathcal L_{2^\omega}$ and
\ref{associatedcrit} we conclude (4).

For (5) by \ref{associatedcrit} it is enough to  have that $K\in \CM_\kappa$ implies
$B_{C(K)^*}\in \CM_\kappa$, and this is exactly 4.3. of \cite{plebanekcorson}.
\end{proof}
However we do not the answer to the following:
\begin{question} Is it consistent\footnote{ Many of the constructions mentioned in this paper are not absolute, that is 
the usual axioms of mathematics (ZFC) are not sufficient to carry them out.
Many follow from additional axioms which were shown to be equiconsistent
with ZFC (they do not lead to contradiction if ZFC does not lead itself) but the consistency of another
group was established only using the method of forcing.
The readers less familiar with these matters should consult when needed, for example, the textbook
\cite{kunen}. This lack of absoluteness is known to be unavoidable as
it  is shown in most cases that some other axioms or forcing
arguments imply the nonexistence of the constructions.} that the class $\mathcal L_{\kappa}$ 
of Lindel\"of Banach spaces in the weak topology of density $\leq \kappa$ is
associated with a class of compact spaces for an uncountable $\kappa$?
\end{question}
By \ref{associatedcrit} one needs to know if $C(B_{X^*})$ is weakly Lindel\"of
if $X$ is. A class of compact spaces as above would need to contain spaces which
are not Corson compact, as for example,
the ladder system space $K$ of
 \cite{pollindelof}.  Another question of a dual sort is the following:

\begin{question} What is the smallest  class $\mathcal K$ of compact spaces
which is associated with a class
 $\mathcal B$ of Banach spaces such that $\mathcal K$ contains the Stone spaces of minimally generated
Boolean algebras of cardinality $2^\omega$?
\end{question}

Minimally generated Boolean algebras were introduced in \cite{koppelbergminimal}. It is shown there
that their Stone spaces contain all dispersed compacta.
But such Stone spaces may not  be sequentially compact, so $\mathcal K$ must be bigger
than the class of all Radon-Nikod\'ym compacta of appropriate weight. See also \cite{borodulin, kunenminimal}.  As similar question 
for scattered spaces seems also open:

\begin{question} What is  the smallest class $\mathcal K$ of compact spaces
which is associated with a class
 $\mathcal B$ of Banach spaces such that $\mathcal K$ contains all
scattered compact spaces of uncountable weight $\kappa$?
\end{question}

This is related to Problem 24 of \cite{antonioracsam} where among others it is asked if every
Radon-Nikod\'ym compact space appears as a subspace of the dual ball $B_{C(K)^*}$
where $K$ is a scattered compact. Our questions seems to be a much weaker version of this question.
\ref{associationlist2} (3)  suggest also the following:

\begin{question} Characterize internally in ZFC the class of Banach spaces whose dual balls
are Corson compacta with property $M$.
\end{question}

\section{Types of universality}

\begin{definition} Let $\mathcal B$ be a class of Banach spaces.
A Banach space $X\in \mathcal B$ is said to be injectively isomorphically (isometrically) universal
for $\mathcal B$ if and only if for every $Y\in \mathcal B$ there is an isomorphic (isometric)
embedding of $Y$ into $X$.\par
\noindent A Banach space $X\in \mathcal B$ is said to be surjectively isomorphically (isometrcally) universal
for $\mathcal B$ if and only if  for every $Y\in \mathcal B$ there is a closed subspace
$Z\subseteq X$ such that $X/Z$ is isomorphic (isometric) to $Y$.
\end{definition}

Note that $X/Z$ is isomorphic to $Y$ for some closed subspace $Z\subseteq X$ if and only if
there is a linear surjection $T:X\rightarrow Y$.\par

\begin{definition} Let $\mathcal K$ be a class of compact Hausdorff spaces.
A compact space $K\in \mathcal K$ is said to be surjectively (injectively) universal for $\mathcal K$
if and only if for every $L\in \mathcal K$ there is a continuous surjection of
$K$ onto $L$ ($K$ maps onto every $K\in\mathcal K$). 
\end{definition}

We will omit some of the adjectives: for a Banach space universal means injectively isomorphically
universal, isometrically universal means injectively isometrically universal,
surjectively universal means isomorphically surjectively universal;  for compact space
universal means surjectively universal.

\begin{proposition}\label{CKtransfer} Let   $\mathcal B$  and $\mathcal K$ be associated classes of Banach spaces
and compact spaces respectively.
Suppose $K$ is  universal for $\mathcal K$, then
$C(K)$ is isometrically universal for $\mathcal B$.

\end{proposition}
\begin{proof}  Let $X\in \mathcal B$, and so by the association \ref{associated} we have $B_{X^*}\in \mathcal K$
 and  then there is a continuous surjection
from $K$ onto $B_{X^*}$, and so $C(B_{X^*})$ is isomorphic to a subspace of $C(K)$ by \ref{semistone}.
As $X$ is isomorphic to a subspace of $C(B_{X^*})$ by \ref{semistone}, we get that $C(K)$ is isometrically universal for $\mathcal B$.
\end{proof}

Under the above assumptions on $\mathcal B$ and $\mathcal K$ it is not true that if 
 $K$ is injectively universal for $\mathcal K$, then
$C(K)$ is surjectively universal for $\mathcal B$. This perhaps explains less interest
in these properties, at least in the context of the duality. Indeed, there cannot be a 
linear bounded surjection
$T: C(K)\rightarrow l_1$  for some compact $K$ because it would be a complemented
subspace of such a space (Theorem VII.5 of \cite{diestel})
 and it would contradict classical results of Pe\l czy\'nski from \cite{projections}.
It follows that there is no surjectively universal Banach space
of the form $C(K)$ for  the class of separable Banach spaces, and of course
$[0,1]^\N$ is an injectively universal compact space for metrizable compact spaces.
The following is also a classical result:

\begin{theorem} $l_1(\kappa)$ is surjectively universal for $\mathcal B_\kappa$.
\end{theorem}

We may consider some weaker versions of universality relevant in the context of
compact spaces.

\begin{definition} 
 Let $\mathcal B$ be a class of Banach spaces.
A Banach space $X\in \mathcal B$ is said to be weakly universal
for $\mathcal B$ if and only if  every $Y\in \mathcal B$ is isomorphic to a quotient of a   subspace
of $X$.\par
\par
\noindent Let $\mathcal K$ be a class of compact Hausdorff spaces.
A compact space $K\in \mathcal K$ is said to be weakly universal for $\mathcal K$
if and only if  every $L\in \mathcal K$ is a  continuous image of
a closed subspace of $K$.
\end{definition}

\begin{proposition}\label{weaklyequiv}
 Let $\mathcal B$ be a class of Banach spaces.
A Banach space $X\in \mathcal B$ is  weakly universal
for $\mathcal B$ if and only if  every $Y\in \mathcal B$ is 
isomorphic to a subspace  of a   quotient of $X$.\par
\par
\noindent Let $\mathcal K$ be a class of compact Hausdorff spaces.
A compact space $K\in \mathcal K$ is  weakly universal for $\mathcal K$
if and only if every $L\in \mathcal K$ is a closed subspace of a continuous image of  $K$.
\end{proposition}
\begin{proof}
If $f:A\rightarrow B\supseteq C$ and $f$ is a surjection, then
$f\restriction f^{-1}[C]\rightarrow C$ is a surjection and $A\supseteq f^{-1}[C]$
so we have one way implications.

To obtain the other implications we need to use the existence of injective
envelopes or Banach spaces or Gleason spaces of compact spaces. 
Given a surjection $f:B\rightarrow C$ and $B\subseteq A$
we can extend $f$ to $f':A\rightarrow C'$ where $C'$ is the injective envelope
in the case of Banach spaces or the Gleason space of $C$ in the case of compact spaces.
Now $C\subseteq f'[A]$ and $f'$ is onto $f'[A]$.
\end{proof}

\begin{proposition}\label{holsztynski}  Let   $\mathcal B$  and $\mathcal K$ be associated classes of Banach spaces
and compact spaces respectively.\par
\begin{enumerate}
\item If $C(K)$ is isometrically  universal for $\mathcal B$, then $K$ is weakly universal for $\mathcal K$.
\item If $K$ is weakly universal for $\mathcal K$, then $C(K)$ is weakly universal for $\mathcal B$.
\end{enumerate}
\end{proposition}
\begin{proof} Let $L\in \mathcal K$. We can use the Holszty\'nski theorem
(see e.g \cite{semadeni}) which says that an isometry $T:C(L)\rightarrow C(K)$
yields a continuous map from a closed subspace of $K$ onto $L$ as required.

Now, let $X\in\mathcal B$. We have that $B_{X^*}\in \mathcal K$ and so there is a closed subspace $K'$ of
$K$ which maps onto $B_{X^*}$. By \ref{dualities} and \ref{semistone} we have $X$ in $C(B_{X^*})$ in $C(K')$
which is  a quotient of $C(K)$ as required.
\end{proof}

The proposition below is used essentially in couple of important papers concerning the
non-existence of universal Banach spaces (\cite{stevofunktor}, \cite{argyrosben}): Having proved that there is no weakly universal compactum,
we conclude that there is no universal Banach space. It seems that it is the only 
 topological method used in the literature of proving the
nonexistence of universal Banach spaces.

\begin{proposition}\label{weaklyutrick} Suppose that $\mathcal B$ and $\mathcal K$ are associated classes
of Banach spaces and compact spaces respectively.
$X$ is weakly universal for $\mathcal B$, if and only if
$B_{X^*}$ is weakly universal for $\mathcal K$. 
\end{proposition}
\begin{proof}
Let $K\in \mathcal K$. Since $X$ is weakly universal for $\mathcal B\ni C(K)$, there is 
a subspace $Y$ of $X$ and a surjective operator
 $T: Y\rightarrow C(K)$.   
Restrictions of functionals from $B_{X^*}$ defined on $X$ to  $Y$ give a continuous
surjection from $B_{X^*}$ onto $B_{Y^*}$ by \ref{dualities}. But 
$T^*$ continuously embeds $B_{C(K)^*}$ into $||T||B_{Y^*}$. Using the fact that
 $K$ is homeomorphic  to a subspace of $B_{C(K)^*}$ by \ref{semistone} 
we get that $K$ is homeomorphic to a subspace of  $B_{Y^*}$ which is a 
a continuous image of $B_{X^*}$ as required.

If  $Y\in \mathcal B$, then $B_{Y^*}$ is a continuous image of a subspace $L$ of $B_{X^*}$ 
so $C(B_{Y^*})$ is isometric to a subspace of $C(L)$ which is a quotient of $C(B_{X^*})$ by \ref{dualities}.
$Y$ s isometric to a subspace of $C(B_{Y^*})$ by \ref{semistone}, so $Y$ is isometric to a subspace of 
 $C(L)$ which is a quotient of $C(B_{X^*})$.
\end{proof}

Weakly universal above cannot be replaced by  universal:
$C([0,1])$ is universal for separable Banach spaces but neither $[0,1]$ 
nor $B_{C([0,1])^*}$ are  surjectively universal for metrizable compact spaces because
they do not map on disconnected spaces. Finally let us make a few remarks on how little is known on
the the relation between   isomorphically universal  and isometrically universal Banach spaces in the nonseparable case. First note the following:

\begin{proposition} There is a separable Banach space which is isomorphically
universal for separable Banach spaces and is not
isometrically universal for separable Banach spaces.
\end{proposition}
\begin{proof} Consider a strictly convex renorming ($||x+y||=2$ for $||x||=||y||=1$ implies $x=y$) of $C([0,1])$ (which exists by Theorem 9 of \cite{clarkson}) and notice that this space continues to be an isomorphically universal space but cannot isometrically include spaces whose norm is not strictly convex.
\end{proof}

However we do not know a corresponding result for nonseparable Banach spaces:

\begin{question}\label{isoiso} Is every universal Banach space of density $\leq 2^\omega$
isomorphic to an isometrically universal Banach space of density $\leq 2^\omega$?
Is it isometrically universal itself?
\end{question}

This question is also open for any of the classes of nonseparable Banach spaces we consider
in this paper.
Actually it is known that for example $\ell_\infty/c_0$ does not have
a strictly convex renorming (\cite{bourgain}).
In particular we do not know the following
general:

\begin{question} Is there a property $P$ of norms such that:
\begin{itemize}
\item If $(X, || ||)$ has $P$ and $Y$ is a closed subspace of $X$, then $(Y, || ||\restriction Y)$
has $P$ as well
\item Not all norms on Banach spaces of density $\leq 2^\omega$  have $P$
\item $C(K)$ spaces have equivalent norm with property $P$
\end{itemize}
\end{question}

The above may not be a difficult question but perhaps $P$  as above is not natural in the context of
the renorming theory. This is because we usually seek as good renormings as possible for a restricted class of
Banach spaces and above we want a bad
renorming for all spaces. A renorming of isometrically universal Banach space with property $P$
would give a universal space which is not an isometrically universal Banach space. However, there are some limitations here, for example any
Banach space isomorphic to $\ell_\infty/c_0$ contains a subspace isometric to $\ell_\infty$ (\cite{part}).
As the existence of universal nonseparable Banach spaces is undecidable
problem for many classes (see the next section), the above issue 
of isometrically universal versus isomorphically universal can be expressed in a much stronger way.
One example is the following:

\begin{question} Are any of the statements below equivalent in ZFC:
\begin{itemize}
\item there is a  universal Banach space of density $\leq 2^\omega$,
\item  there is an
isometrically  universal Banach space of density $\leq 2^\omega$,
\item there is  a  universal compact space of weight $\leq 2^\omega$?
\end{itemize}
\end{question}

This question is also open for any of the classes of nonseparable Banach spaces we consider
in this paper.
M. Krupski and W. Marciszewski constructed in   \cite{krupski} the first consistent  example of a Banach space which isomorphically 
embeds in $l_\infty/c_0$ but does not embed isometrically. 
This does not answer Question \ref{isoiso} because in the model considered in \cite{krupski}
the space  $l_\infty/c_0$ is not universal.

\section{The existence and the non-existence of universal spaces}

As seen in the previous section we have general results (\ref{CKtransfer}) allowing us to obtain universal Banach spaces
from universal compact spaces.
All positive  universality results for Banach spaces which we include here come from the universality of
compact space via the arguments described above. In some cases they could be also obtained by general model-theoretic
arguments, but as far as we know, in all these cases we can also build the corresponding universal compact space.

 However the negative universality results
require usually some other ideas, unless we can prove that there is no weakly universal compact space of our class 
like in \cite{argyrosben, stevofunktor} and then apply \ref{weaklyutrick}.
The other ideas exploited in the relevant literature come from graphs (\cite{bell}), ordinal indices in the spirit
of Cantor-Bendixon height (\cite{szlenk}, \cite{wojtaszczyk}, \cite{asplunduniversal}),
or forcing genericity(\cite{universal}, \cite{ug}).

\subsection{Uniformly Eberlein compacta and Hilbert generated Banach spaces.}  First we note that the definition of
a uniform Eberlein compactum is so restrictive that  it is not surprising that
there are weakly universal objects in this class.

\begin{proposition}$B_{l_2(\kappa)}$ is  injectively universal for  $\UE_\kappa$. In particular 
$C(B_{l_2(\kappa)})$ is weakly universal for  $\overline{\mathcal{H}}_\kappa$.
\end{proposition}
\begin{proof}
Note that by taking appropriate quotients,  every uniform Eberlein compact space of weight $\kappa$
can be continuously embedded in the Hilbert space $l_2(\kappa)$ with the weak topology
and the image of this embedding  must be a  bounded set. As the unit ball
in a Hilbert space is weakly compact,  it is  injectively universal, and so weakly universal uniform Eberlein compact space. 
By \ref{weaklyutrick}, this gives that $C(B_{l_2(\kappa)})$ is a weakly universal in  $\overline{\mathcal{H}}_\kappa$.
\end{proof}

A result of Benyamini, Rudin, Wage provides another weakly universal compact space.

\begin{proposition}[\cite{BRW}]\label{wuuniforme}
${A(\kappa)}^\N$ is weakly universal in $\UE_{\kappa}$.
In particular $C({A(\kappa)}^\N)$ is  weakly universal for $\overline{\mathcal{H}}_\kappa$.
\end{proposition}
\begin{proof} Apply \ref{holsztynski} (2).
\end{proof}

However, we have the following result of M. Bell:

\begin{theorem}[\cite{bellpolyadic}]\label{bellpolyadic} There is a compact $K\in \UE_{\omega_1}$
which is not a continuous image of any space ${A(\kappa)}^\N$.
In particular the space
${A(\kappa)}^\N$ is   not universal in $\UE_{\kappa}$ for
an uncountable $\kappa$.
\end{theorem}

A. Avil\'es proved in \cite{antoniohilbert} that $B_{l_2(\kappa)}$ is a continuous image of $A(\kappa)^\N$. Moving from
weakly universal objects to universal objects we encounter undecidability of their existence.

\begin{theorem}\label{negativebell}\cite{bell} It is consistent
  that there is
no universal uniform Eberlein compact space of weight $\omega_1$
\end{theorem}

The argument used above employs the consistency of the nonexistence of 
a universal graph of cardinality $\omega_1$ proved by S. Shelah (\cite{shelahgraphs1, shelahgraphs2}).
It turns out one can associate graphs with uniform Eberlein compacta (see\cite{bell}).  See Question 
\ref{questiongraph}.

Under some cardinal arithmetics hypothesis, using a general procedure applicable in other categories,
  M. Bell also proved a positive result for compact spaces 
which gives, by \ref{CKtransfer}, a positive result for associated class of Banach spaces:

\begin{theorem}\cite{bell} Assume $\kappa=2^{<\kappa}$.  Then, there 
 is  a universal uniform Eberlein compact space  $UE_\kappa$ in $\UE_{\kappa}$
and so $C(UE_\kappa)$ is isometrically universal Banach space
 for $\overline{\mathcal{H}}_\kappa$.
In particular, the above results hold for $\kappa=2^\omega=\omega_1$ if we
assume the continuum hypothesis.
\end{theorem}

The negative result \ref{negativebell} of Bell cannot be directly applied via \ref{weaklyutrick} for
the class of Banach spaces $\overline{\mathcal{H}}_\kappa$
because  weakly universal compact spaces exist here (\ref{wuuniforme}).
C.~Brech and the author obtained the Banach space version of this negative result 
employing a forcing genericity argument.
Actually the result is much stronger and has other universal consequences which we will mention later:

\begin{theorem}\label{ugresults}\cite{ug} It is consistent that there is no Banach space of density $2^\omega$ or $\omega_1$  
which contains isomorphic copies of all Banach spaces from  $\UE_{2^\omega}$
or from $\UE_{\omega_1}$ respectively. In particular, it is consistent that there is no universal Banach space 
neither in $\overline{\mathcal{H}}_{\omega_1}$ nor in $\overline{\mathcal{H}}_{2^\omega}$
\end{theorem}

The question of the possible nonexistence of universal spaces in $\overline{\mathcal{H}}_\kappa$ 
for $\kappa$ different than $\omega_1$ or $2^\omega$ were not addressed in the literature yet.
There are other negative results of universal Eberlein compact spaces, e.g. in \cite{mirna}, which were
not dualized to Banach spaces yet.

\vskip 15pt

\subsection{Eberlein compacta and weakly compactly generated Banach spaces.}
It is well known that all Eberlein compacta are subspaces of dual balls of reflexive Banach spaces with
the weak$^*$ topology. However this class of compact spaces is much reacher than the dual balls of Hilbert spaces and
 it is not enough to obtain weakly universal objects as in the case of uniform Eberlein compacta.

\begin{theorem}\cite{argyrosben}
If $\kappa^\omega=\kappa$ or $\kappa=\omega_1$ then there is no weakly
universal Eberlein compact of weight $\kappa$ nor a universal WCG Banach space
of density $\kappa$. If $\kappa$ is a strong limit cardinal of countable cofinality, then there is
a universal Eberlein compact of weight $\kappa$, and so,  there is a universal WCG Banach space
of density $\kappa$.
\end{theorem}

Here  S. Argyros  and Y. Benyamini could conclude the negative result for the associated class 
of Banach spaces  via \ref{weaklyutrick}.
Thus under G.C.H. the situation is completely settled. However many consistent cardinal arithmetics are
not covered by the above result. It seems that the following
questions were not settled in the literature:

\begin{question} Is it consistent that there is no universal space in $\E_{\omega_\omega}$, in $\mathcal{WCG}_{\omega_\omega}$? 
\end{question}

\begin{question}Is it consistent that there is  a universal space in $\E_{\omega_2}$, in $\mathcal{WCG}_{\omega_2}$? 
\end{question}

There are also results on universal Eberlein compacta from
 more restrictive classes, e.g. in \cite{bellwitek}, which did not find
their Banach space theory versions as yet.

\vskip 15pt
\subsection{Corson compacta and weakly Lindel\"of determined and weakly Lindel\"of Banach spaces.} S. Todorcevic proved the following:

\begin{theorem}\label{nonunivwld}\cite{stevofunktor}
For every countably tight compact space $K$ of weight $\leq 2^\omega$
there is a Corson compact $K'$ of weight $\leq 2^\omega$ where all Radon  measures have
separable supports and which is not a continuous image of any closed set of $K$.
In particular there is no universal weakly Lindel\"of determined Banach spaces   in $\mathcal{WLD}_{2^\omega}$.
\end{theorem}

Knowing that every Corson compact space has
countable tightness, the result implies that there is no Corson compact of weight
continuum whose closed subspaces  map onto all Corson compacta with property M of weight continuum.
Using (1) and (2) of \ref{associationlist2} as in \ref{weaklyutrick}, if there were
a universal Banach space $X$  in $\mathcal{WLD}_{2^\omega}$, then $B_{X^*}$
were a Corson compact whose subspaces  mapped onto all Corson compacta with  property M:
Given $K\in \CM_\kappa$, $C(K)\in \mathcal{WLD}_{2^\omega}$, so $C(K)$ isomorphically
embeds in $X$ and so $B_{X^*}$ maps onto $B_{C(K)^*}$ by \ref{dualities} which includes $K$ by \ref{semistone}.

 In Section 4 of  \cite{stevofunktor} we have the following

\begin{theorem} There is no universal weakly Lindel\"of Banach space of density $2^\omega$.
\end{theorem}

As we do not have (1), (2) of \ref{associationlist2} 
for the class of weakly Lindel\"of Banach spaces  an approach  presented in Section 4 of \cite{stevofunktor}
is necessarily different than that leading to the \ref{nonunivwld}.

\begin{question} 
\begin{enumerate}
\item Are there (consistently) universal Banach spaces in ${WLD}_{\kappa}$
for $\kappa\not=2^\omega$, for example consistently
for $\kappa=\omega_1$? 
\item  Are there (consistently) universal Corson compact spaces
(where all Radon measures have separable supports) spaces of weight $\leq \kappa$
for $\kappa\not=2^\omega$, for example consistently
for $\kappa=\omega_1$?
\item Are there (consistently) universal Banach spaces in ${L}_{\kappa}$
for $\kappa\not=2^\omega$, for example consistently
for $\kappa=\omega_1$? 
\end{enumerate}
\end{question}
\vskip 15pt

\subsection{(Quasi) Radon-Nikod\'ym compacta and (subspaces of) Asplund generated Banach spaces}
Perhaps one should start this subsection with looking at the class of scattered compact spaces whose
spirit is present in the classes considered here. One has the Cantor-Bendixon height $ht(K)$ for all compact
scattered spaces $K$ so that $ht(K)$ is an ordinal such that $|ht(K)|$ is not bigger than the weight of $K$, 
if $L$ is a closed subset of $K$ or if $L$ is a continuous image of $K$, then $ht(L)\leq ht(K)$, and moreover
for every cardinal $\kappa$ and every $\alpha<\kappa^+$ there is a compact scattered space 
 $K$ of weight $\kappa$ such that
$ht(K)>\alpha$. This directly implies that there is no weakly universal scattered compact space of any fixed weight.

This idea was first transfered in the realm of Banach spaes by Szlenk in \cite{szlenk} who
defined a similar index for the dual balls of separable reflexive Banach spaces obtaining the
nonexistence a universal reflexive Banach space. This could be generalized to  to all Asplund spaces:

\begin{theorem}[\cite{asplunduniversal}, \cite{wojtaszczyk}] There is no universal Banach space for the class of Asplund spaces of density character $\kappa$
for any infinite cardinal $\kappa$
\end{theorem}

The dual balls of Asplund spaces allow to define a version of Cantor-Bendixon height.
However the class of Asplund spaces is not associated with a class of compact spaces:
if we were to have $C(K)$ Asplund, $K$ would be scattered but dual balls are always connected, so
never scattered. Radon-Nikod\'ym compacta are exactly the weakly$^*$ compact subspaces of
dual balls of Asplund spaces and are associated with the class of Asplund generated Banach spaces.
Thus we can also define Cantor-Bendixon or Szlenk index on them, however it depends on
a particular representation as a weakly$^*$ compact subspace of the dual ball of an Asplund space.
It creates the difficulty which probably explains why
 the question of universal object in these classes was not addressed in the literature. 
As Eberlein and uniform Eberlein compacta are Radon-Nikod\'ym, it
 follows from results of \cite{ug} i.e., \ref{ugresults} that it is consistent that there are no universal
objects in $\RN_{2^\omega}$,  $\RN_{\omega_1}$, $\QRN_{2^\omega}$, $\QRN_{2^\omega}$,
$\mathcal{AG}_{2^\omega}$,  $\mathcal{AG}_{\omega_1}$, $\mathcal{SAG}_{2^\omega}$,  $\mathcal{SAG}_{\omega_1}$.

The fragmentability of (quasi) Radon-Nikod\'ym compacta
mentioned above suggest, however the possibility of introducing  an ordinal index
which could be used for proving the nonexistence of universal spaces like Cantor-Bendixon height of a scattered space can be used to
prove the nonexistence of a universal metrizable scattered compact space.  Remarks 
concerning this topic are implicitly used in \cite{argyrosben} p. 308
or in \cite{asplunduniversal} p. 2034. Here we make this problem explicit taking into account the result of \cite{rn}:

\begin{question} Is it possible to associate to each quasi Radon-Nikod\'ym compact $K$ an ordinal index $i(K)$ having the
following properties:
\begin{itemize}
\item $|i(K)|$ is not bigger than the weight of $K$, 
\item If $L$ is a closed subset of $K$ or if $L$ is a continuous image of $K$, then $i(L)\leq i(K)$,
\item For every $\alpha<\kappa^+$ there is a quasi Radon-Nikod\'ym compactum $K$ of weight $\kappa$ such that
$i(K)>\alpha$.
\end{itemize}
\end{question}

In \cite{argyrosben} and index having the properties above was introduced for Eberlein compacta.
To avoid the dependence of the representation of a given Eberlein compact space
the authors exploited the possibility of embedding them as weakly compact subspaces in $c_0(\kappa)$.
However Radon-Nikod\'ym compacta which are not Eberlein compact do not admit such embeddings.
It is clear that the positive answer to the above question would give the negative answer to:

\begin{question} Is it consistent that there are universal spaces in one of the classes $\RN_{2^\omega}$,  $\RN_{\omega_1}$, $\QRN_{2^\omega}$, $\QRN_{2^\omega}$?
\end{question}

\begin{question} Is it consistent that there are universal spaces in one of the classes $\mathcal{AG}_{2^\omega}$,  $\mathcal{AG}_{\omega_1}$, $\mathcal{SAG}_{2^\omega}$,  $\mathcal{SAG}_{\omega_1}$?
\end{question}

\vskip 15pt
\subsection{All compact spaces and Banach spaces of a given size}

Under many additional set-theoretic assumptions there exist  universal compact spaces of a given weight and so 
applying \ref{CKtransfer} one
can obtain universal Banach spaces of a given density.

\begin{theorem}[\cite{volpin}] Assume GCH.
Then there is a universal compact space  in $\K_{\kappa}$ for each cardinal number $\kappa$
and so there is isometrically universal Banach space in each  $\mathcal B_{\kappa}$.
\end{theorem}

\begin{theorem}[\cite{parovicenko}]\label{parovicenko} Assume CH. Then $\N^*$ is universal in $\K_{2^\omega}$
and so $l_\infty/c_0$ is isometrically universal in $\mathcal B_{2^\omega}$.
\end{theorem}

However, the above results require some additional set theoretic assumptions as seen from the following theorems.

\begin{theorem}[\cite{dowhart}] It is consistent that there is no
universal compact space for $\K_{2^\omega}$.
\end{theorem}

The consistency of the nonexistence of a universal space in $\K_{\omega_1}$ follows from the next result 
and \ref{CKtransfer} but
it was probably known before.

\begin{theorem}[\cite{universal}]\label{universal} The existence of  an isomorphically universal Banach space in $\mathcal B_{2^\omega}$
or in $\mathcal B_{\omega_1}$
is undecidable. 
\end{theorem}

Surjectively universal Banach spaces exist without any special set-theoretic assumptions for
any density character $\kappa$. It is a classical result that $\ell_1(\kappa)$ is such a space.
We can also obtain weak universal Banach spaces of the form $C(K)$  without any special set-theoretic assumptions for
any density character. This is a consequence of the existence of injectively universal
compact spaces:

\begin{proposition} Let $\kappa$ be a cardinal.
$[0,1]^\kappa$ is injectively universal in $\K_{2^\omega}$. In particular, any space
that maps onto $[0,1]^{2^\omega}$, and so, 
$N^*$, is weakly universal in $\K_{2^\omega}$. In particular $\ell_\infty/c_0$ is weakly
universal for $\mathcal{B}_{2^\omega}$.
\end{proposition}
\begin{proof} Clearly any compact space of weight $\kappa$ embeds into
$[0,1]^\kappa$. $N^*$ maps onto $[0,1]^{2^\omega}$ because $\wp(\N)$ has
an independent family of infinite sets which gives a continuous surjection of
$\N^*$ onto $2^{2^\omega}$. Using the standard continuous mapping
of the Cantor set onto $[0,1]$ we obtain the desired mapping of $\N^*$ onto  $[0,1]^{2^\omega}$.
To go to the Banach spaces use \ref{weaklyutrick}.
\end{proof}

However, there are no surjectively universal Banach spaces of the form $C(K)$ as noted after
\ref{CKtransfer}. Let us finish this section with a different kind of a problem.  There are many
results about the existence and nonexistence of universal structures obtained by methods from
 a field of mathematical logic which is called model theory which deals with structures, languages
and theories expressible in these languages which are true in theses structures (models of theories).
There is a big limitation of applications of model theory in the theory of Banach spaces (or even topology
of compact spaces). The Banach space theory is not a first order theory which is best understood and investigated type
of a language in model theory.  It may happen however that the existence of universal Banach spaces can be
equivalent to the existence of simpler structures describable in terms of a first order language (which would code Banach spaces). Links between universality of Boolean algebras (a first order structure) and Banach spaces
are well know  (see e.g., 1.1. \cite{universal}) but as far as now equivalences were not obtained.
Also the existence of universal graphs, which are first order structures is implied  by
the existence of universal 
uniform Eberlein compact spaces as shown at the end of \cite{bell}. Of course any equivalence
would give a possibility of transferring model-theoretic results to questions concerning universal Banach spaces.
One concrete question could be:

\begin{question}\label{questiongraph} Is it consistent that there is a universal graph 
of size $\omega_1$ ($2^\omega$) but there is no universal Banach space for 
$\mathcal B_{\omega_1}$ ($\mathcal B_{2^\omega}$)?
\end{question}

It should be noted that there are some interesting results where first order model theory is used
to conclude the universality results concerning Banach spaces like for example in the work
of Shelah and Usvyatsov (\cite{shelahusvyatsov}) where it was first proved that it is consistent that there is
no isometrically universal Banach space for $\mathcal B_{2^\omega}$.
However these results concern only the isometric theory as far as now.

\section{The non-universality of $l_\infty/c_0$ and $\N^*$}

The spaces $\ell_\infty/c_0$ and $\N^*$ play a special role in the classes $\mathcal B_{2^\omega}$
and $\K_{2^\omega}$. One of the reasons is the universality result of \ref{parovicenko}. Of course
$\ell_\infty/c_0$ is isometric to $C(\N^*)$.
After K. Kunen proved that $[0, 2^\omega]$ may not be a continuous image of
$\N^*$ the question  under which assumptions which compact spaces must be continuous images of $\N^*$
was considered by many authors. Recently the dual question, which Banach spaces must embed isomorphically
or isometrically in $\ell_\infty/c_0$
was considered be several authors as well. In this section we survey these results.
It is not difficult to see the following well known facts:

\begin{proposition}All separable compact spaces are continuous images 
of $\N^*$ and all Banach spaces whose dual balls are weakly separable isometrically embed
in $\ell_\infty/c_0$.
\end{proposition}
 \begin{proof} Consider a separable compact $K$. Let $L$ be a totally disconnected separable 
compact space which maps onto $L$ (for example take an a totally disconnected $L$ which maps irreducibly onto $K$,
like a Gleason space \cite{engelking} 6.3.19). By sending  the dense countable subset of $L$ onto $\N$ we can embed the
Boolean algebra of clopen subsets of $L$ into $\wp(\N)$,
 so by the Stone duality $L$ is a continuous image of $\beta\N$
which is a continuous image of $\N^*$. 

If $X$ has separable dual ball, then $\N^*$ maps onto $B_{X^*}$, so $\l_\infty/c_0\equiv C(\N^*)$ isometrically
embeds  $C(B_{X^*})$ which isometrically embeds $X$ by \ref{semistone}.
\end{proof}

Note that the class of Banach spaces with weakly$^*$ separable duals is not closed
under isomorphisms (Ex. 12.40 \cite{fabianetal}) and includes
all HI spaces.  In \cite{przymusinski} Przymusi\'nski  proved that all perfectly normal 
compact spaces are continuous images of $\N^*$. 
 The topological part of the following theorem of van Douwen and
Przymusi\'nski already can be dualized to Banach spaces:

\begin{theorem}[\cite{przymusinskivd}] Assume MA$+\neg$CH. Let $\kappa<2^\omega$. Then each compact
$K\in \K_\kappa$ is a continuous image of $\N^*$. In particular each Banach space
of density character $\kappa<2^\omega$ isometrically embeds into $\ell_\infty/c_0$.
\end{theorem}

Actually, the Parovichenko  property of $\wp(\N)/Fin$ gives that every Boolean algebra
of cardinality not bigger than $\omega_1$ embeds into $\wp(\N)/Fin$ which by the Stone duality and
\ref{dualities} gives:

\begin{proposition}
Every compact space in $\K_{\omega_1}$ is a continuous image of $\N^*$, every
Banach space in $\mathcal B_{\omega_1}$ embeds isometrically in $\ell_\infty/c_0$.
\end{proposition}

Polyadic space are by definition continuous images of products of the form
$A(\kappa)^{\tau}$. They are the smallest class of topological spaces containing
metric compact spaces and closed under products and continuous images, they also contain all
dyadic spaces i.e., continuous images of $2^\kappa$ (see \cite{gerlits}). We have the following:

\begin{theorem}[\cite{bellshapiro}]\label{polyadic} Every polyadic space $K$ of weight $\leq 2^\omega$ is a continuous
image of $\N^*$. In particular $C(K)$ for such a $K$ isometrically embeds in $l_\infty/c_0$.
\end{theorem}

Note  however  that $C(K)$s may not embed into $l_\infty/c_0$
for some  continuous images $K$ of polyadic spaces by \ref{wuuniforme} and \ref{ugresults}.
This completes our list of absolute results. It follows from our results with C.~Brech from \cite{universal} (see \ref{universal})
that it is consistent that some Banach spaces of density $2^\omega$ do not isomorphically embed in $\ell_\infty/c_0$.
In \cite{universal} we also proved the following:

\begin{theorem}\cite{universal} It is consistent that there exist universal
Banach spaces for $\mathcal B_{2^\omega}$ but $l_\infty/c_0$ is not among them.
\end{theorem}

The proof of the above results consist of showing that $C([0,2^\omega])$
does not embed isomorphically into $l_\infty/c_0$ in a suitable model of set theory dualizing the result of Kunen.

\begin{question} Is it consistent that $[0,2^\omega]$ is not a continuous image of
$\N^*$ but $C([0,2^\omega])$ isomorphically embeds in $\ell_\infty/c_0$?
\end{question}

At the end of \cite{universal} we describe how to construct a zero-dimensional compact $K$
such that $[0,2^\omega]$ is not a continuous image of
$K$ but $C([0,2^\omega])$ isomorphically embeds in $C(K)$ by we do not know the answer
to the following:

\begin{question}Is it consistent that $C([0,2^\omega])$ isomorphically embeds
in $l_\infty/c_0$ but $[0,2^\omega]$ is not a continuous image of $\N^*$?
\end{question}

Let us mention that 
M. Krupski and W. Marciszewski prove in \cite{krupski} 
that consistently there could be a uniform Eberlein compact $K$ such that
$C(K)$ embeds isomorphically in $\ell_\infty/c_0$ but does not embed isometrically.
 There is another example concerning similar issues. In \cite{dowhartmeasure}
it is proved that the Stone space $K$ of the measure  algebra (Boolean algebra of Borel sets of $[0,1]$ divided by
Lebesgue measure sets zero) consistently may not be a continuous image of $\N^*$. However it
is well known that $C(K)$ is isometric to $L_\infty([0,1])$ which is isomorphic to $\ell_\infty$ (by a theorem
of Pe\l czy\'nski) which is isomorphic to a complemented subspace of  $\ell_\infty/c_0$. This generate the following:

\begin{question} Suppose that $K$ is the Stone space  of the  measure  algebra (Borel sets of $[0,1]$ divided by
Lebesgue measure sets zero). Is it consistent that $C(K)$ does not isometrically embed into 
$\ell_\infty/c_0$?
\end{question}

By Holszty\'nski's theorem and the Stone duality
it would be enough to prove that the  measure  algebra does not embed into
any quotient of $\wp(\N)/Fin$.
In \cite{stevoemb} S. Todorcevic presented an axiomatic approach to 
the problem of isomorphic (and in fact isometric) embeddings into $l_\infty/c_0$. To express
these results we need a few definitions: $\wp^n(\kappa)$ denotes the $\sigma$-field 
of sets generated by sets of the form $A_1\times ... \times A_n$ for $A_1, ..., A_n\subseteq \kappa$.
If $E\subseteq \kappa^2$ is a binary relation, then $E^{[n]}=\{(x_1,..., x_n)\in \kappa^n: \forall i< j\ (x_i, x_j)\in E\}$.
We will say that $\kappa$ is an $n$-Kunen cardinal if and only if for every
binary relation $E\subseteq \kappa^2$ the sets $E^{[n]}$ and $(\kappa^2\setminus E)^{[n]}$ can be separated by
an element of $\wp^n(\kappa)$. It can be proved that $2$-Kunen cardinal is a Kunen cardinal in the sense that
$\wp(\kappa\times\kappa)=\wp^2(\kappa)$. More information on Kunen cardinals can be found in \cite{antoniokunen}.
In particular it is known (and implicit in \cite{kunenthesis}) that it s consistent that $2^\omega$ is not $n$-Kunen cardinal for
any $n\in \N$.

\begin{theorem}\cite{stevoemb} If $2^\omega$ is not an $n$-Kunen cardinal for any $n\in\N$, then
there is a first countable compact space $K$ and a Corson compact space $L$ 
such that neither $C(K)$ nor $C(L)$ isomorphically embed into $l_\infty/c_0$.
\end{theorem}

The Corson part of the above result was improved in \cite{krupski}
in the following way:

\begin{theorem}\cite{krupski}
If $2^\omega$ is not an $n$-Kunen cardinal for any $n\in\N$, then
there is a uniform Eberlein compact space $K$ 
such that  $C(K)$ does not embed isomorphically  into $l_\infty/c_0$.
\end{theorem}

This is strictly related to \ref{polyadic} and \ref{wuuniforme}, \ref{bellpolyadic}, i.e., the gap between what
always can be embedded in $l_\infty/c_0$ and what consistently cannot be embedded is quite tight.
It is also related to results of \cite{ug} (see \ref{ugresults}).
  We note that \cite{mirnalajos} also contain consistency results improving the results of \cite{universal} to
Corson compacta and use a more technical approach than in \cite{stevoemb}. 
 We close this section by the following:

\begin{question} Is it consistent that $\N^*$ is not universal for $\K_{2^\omega}$ but
$\ell_\infty/c_0$ is universal for $\mathcal B_{2^\omega}$?
\end{question}

\bibliographystyle{amsplain}

\end{document}